\documentclass[12pt,a4paper,reqno]{amsart}
\usepackage[all]{xy}
\usepackage{amsmath}
\usepackage{amsfonts}
\usepackage{amssymb,amsthm,amsfonts,amsthm,latexsym,enumerate,url,cases}
\numberwithin{equation}{section}
\usepackage{mathrsfs}
\usepackage{hyperref}
\usepackage{extarrows}
\hypersetup{colorlinks=true,citecolor=blue,linkcolor=blue,urlcolor=blue}
\def\pmod #1{\ ({\rm{mod}}\ #1)}

\newtheorem{theorem*}{Theorem}
\newtheorem{lemma*}{Lemma}
\theoremstyle{plain}
\newtheorem{theorem}{Theorem}
\newtheorem{lemma}{Lemma}

\newtheorem{conjecture}{Conjecture}
\theoremstyle{definition}


\begin{document}

\title
[{On a conjecture on shifted primes with large prime factors}] {On a conjecture on shifted primes with large prime factors}

\author
[Yuchen Ding] {Yuchen Ding}

\address{(Yuchen Ding) School of Mathematical Science,  Yangzhou University, Yangzhou 225002, People's Republic of China}
\email{ycding@yzu.edu.cn}
\thanks{*Corresponding author}


\begin{abstract} Let $\mathcal{P}$ be the set of all primes and $\pi(x)$ be the number of primes up to $x$. For any $n\ge 2$, let $P^+(n)$ be the largest prime factor of $n$. For $0<c<1$, let $$T_c(x)=\#\{p\le x:p\in \mathcal{P},P^+(p-1)\ge p^c\}.$$
In this note, we prove that there exists some $c<1$ such that $$\limsup_{x\rightarrow\infty}\frac{T_c(x)}{\pi(x)}<\frac12,$$
which disproves a conjecture of Chen and Chen.
\end{abstract}
\maketitle

\baselineskip 18pt

\section{Introduction}
The investigations of the largest prime factors of shifted primes start from a cute article of Goldfeld \cite{Gold}, where he proved that
$$\#\{p\le x:p\in \mathcal{P},P^+(p-1)>p^{1/2}\}\ge \frac{1}{2}\frac{x}{\log x}+O\left(\frac{x\log\log x}{(\log x)^2}\right),$$
where $\mathcal{P}$ is the set of all primes and $P^+(n)$ denotes the largest prime factor of integer $n$. Conventionally, we set $P^+(1)=1$. For $0<c<1$, let $$T_c(x)=\#\{p\le x:p\in \mathcal{P},P^+(p-1)\ge p^c\}.$$
Goldfeld further remarked that one can show
\begin{align}\label{eq1}
\liminf_{x\rightarrow\infty}\frac{T_c(x)}{\pi(x)}>0
\end{align}
provided that $c<7/12$, where $\pi(x)$ is the number of primes up to $x$. Fouvry \cite{Fou} improved this by showing that there is some $c_0\in(2/3,1)$ such that inequality (\ref{eq1}) holds.

In the same direction, Luca et al. \cite{LMP} considered the lower bound for $T_c(x)$ for small values of $c$. They proved that for $1/4\le c\le 1/2$
$$T_c(x)\ge (1-c)\frac{x}{\log x}+E(x),$$
where $$
E(x)\ll
\begin{cases}
x\log\log x/(\log x)^2, & \text{for~} 1/4<c\le 1/2,\\
x/(\log x)^{5/3}, & \text{for~} c=1/4.
\end{cases}
$$ As pointed out by Chen and Chen \cite{CC}, the method of Luca et al. (essentially due to Goldfeld) cannot be applied to $c\in (0,1/4)$. By some refinements of the argument employed by Luca et al., Chen and Chen could extend $c$ to the interval $(0,1/2)$ with slightly better error terms of the order of magnitude $O(x/(\log x)^2)$. Moreover, Chen and Chen proved the following interesting result. For any integer $k\ge 2$, there exists at most one $c\in [\frac1{k+1},\frac1k)$ such that $$T_c(x)=(1-c)\frac{x}{\log x}+o\left(\frac{x}{\log x}\right).$$ Based on this result, Chen and Chen made the following conjecture.
\begin{conjecture}\label{conje1} For any integer $k\ge 1$ and any $c\in [\frac1{k+1},\frac1k)$, we have
$$T_c(x)\ge \left(1-\frac{1}{k+1}\right)\frac{x}{\log x}+o\left(\frac{x}{\log x}\right).$$
\end{conjecture}
Let $\rho(u)$ be the Dickman function, defined as the unique continuous solution of the equation
differential--difference
\begin{align*}
\begin{cases}
\rho(u)=1, & 0\le u\le 1,\\
u\rho'(u)=-\rho(u-1), & u>1.
\end{cases}
\end{align*}
Let $\theta_1\thickapprox 0.3517$ be the unique solution of equation
$$\theta-4\int_{1/\theta-1}^{1/\theta}\frac{\rho(t)}{t}dt=0.$$
Feng and Wu \cite{FW} proved that
$$T_{c}(x)\ge\left(1-4\int_{1/\theta-1}^{1/\theta}\frac{\rho(t)}{t}dt+o(1)\right)\pi(x)$$
provided $0<c<\theta_1$, which confirmed Conjecture \ref{conje1} for $k\ge 3$ by numerical values involving Dickman's function. Later, the lower bound for $T_c(x)$ was further improved to
$$T_{c}(x)\ge\left(1-4\rho(1/\theta)+o(1)\right)\pi(x)$$
by Liu, Wu and Xi \cite{LWX} provided that $0<\theta<\theta_2\thickapprox0.3734$, where $\theta_2$ is the unique solution to equation $\theta-4\rho(1/\theta)=0$.

In the present note, we show that Conjecture \ref{conje1} is actually incorrect for $k=1$. Precisely,
\begin{theorem}\label{thm1}
There is an absolute constant $c<1$ such that $$\limsup_{x\rightarrow\infty}\frac{T_c(x)}{\pi(x)}<1/2.$$
\end{theorem}

\section{Proofs}
Before presenting the proof of Theorem \ref{thm1}, we first fix some notations and then state a few lemmas to be used later. Let $\Lambda(n)$ be the von Mangoldt function. As usual, $\varphi(n)$ denotes the Euler totient function. Let $\pi(x;b,a)$ be the number of primes $p\equiv a\pmod{b}$ up to $x$. Following Goldfeld and Luca et al., we define
$$L(x;u,v)=\sum_{u<m\le v}\Lambda(m)\pi(x;m,1).$$

The first two lemmas are included in the proof of Goldfeld (see also \cite[Lemma 2.1]{CC}).
\begin{lemma}\label{lem1}\cite[bottom of page 23]{Gold}
For sufficiently large $x$, we have
$$L(x;1,x)=x+O\left(x/\log x\right).$$
\end{lemma}

\begin{lemma}\label{lem2}\cite[equations (2) and (3)]{Gold}
For sufficiently large $x$, we have
$$L(x;1,x^{1/2})=x/2+O\left(\frac{x\log\log x}{\log x}\right).$$
\end{lemma}

The next lemma is another conjecture of Chen and Chen confirmed by Wu later.
\begin{lemma}\label{lem3}\cite[Theorem 2]{Wu}
For $0<c<1$, let $$T'_c(x)=\#\{p\le x:p\in\mathcal{P}, P^+(p-1)\ge x^c\}.$$
Then for sufficiently large $x$ we have
$$T_c(x)=T'_c(x)+O\left(\frac{x\log\log x}{(\log x)^2}\right).$$
\end{lemma}

The proof of our theorem, among other things, is based on the following deep result which appeared first in the paper of Banks and Shparlinski \cite[Lemma 2.1]{BS}. Here, we use the version stated by Wu with a slight adjustment from prime modulus $q$ to integer modulus $m$.
\begin{lemma}\label{lem4}\cite[Lemma 2.2]{Wu}
There exist two functions $K_2(\theta)>K_1(\theta)>0$, defined on the interval $(0,17/32)$ such that for each fixed real $A>0$, and all sufficiently large $Q=x^\theta$, the inequalities
$$K_1(\theta)\frac{\pi(x)}{\varphi(m)}\le \pi(x;m,1)\le K_2(\theta)\frac{\pi(x)}{\varphi(m)}$$
hold for all integers $m\in(Q,2Q]$ with at most $O\left(Q(\log Q)^{-A}\right)$ exceptions, where the implied constant depends only on $a,A$ and $\theta$. Moreover, for any fixed $\varepsilon>0$, these functions can be chosen to satisfy the following properties:\\
$\bullet$ $K_1(\theta)$ is monotonic decreasing, and $K_2(\theta)$ is monotonic increasing.\\
$\bullet$ $K_1(1/2)=1-\varepsilon$ and $K_2(1/2)=1+\varepsilon$.
\end{lemma}
{\it Remark. The above lemma is due to the Bombieri--Vinogradov theorem for $0<\theta<1/2$, Baker--Harman \cite{BH} for $1/2\le\theta\le 13/25$ and Mikawa \cite{Mi} for $13/25\le \theta\le 17/32$.}

Now we turn to the proof of Theorem \ref{thm1}.
\begin{proof}[Proof of Theorem \ref{thm1}] Let $33/64<c<1$ be a fixed constant. Throughout our proof, $p$ and $q$ will always denote primes unless indicated otherwise. The lower bound on $T_c(x)$, as observed by Goldfeld, starts from
$$T_c(x)\ge\sum_{p\le x}\sum_{\substack{x^c\le q\le x\\ q|p-1}}1\ge \frac{1}{\log x}\sum_{p\le x}\sum_{\substack{x^c\le q\le x\\ q|p-1}}\log q.$$
For an upper bound of $T_c(x)$, we begin with the dual observation inspired by Goldfeld with the help of Lemma \ref{lem3}
\begin{align*}
T_c(x)&=\sum_{p\le x}\sum_{\substack{x^c\le q\le x\\q|p-1}}1+O\left(\frac{x\log\log x}{(\log x)^2}\right)
&\le \frac{\log x}{c}\sum_{p\le x}\sum_{\substack{x^c\le q\le x\\q|p-1}}\log q+O\left(\frac{x\log\log x}{(\log x)^2}\right).
\end{align*}
Changing the order of summations above, we get
\begin{align}\label{eq2-1}
T_c(x)\le \frac{\log x}{c}\sum_{x^c\le q\le x}\pi(x;q,1)\log q+O\left(\frac{x\log\log x}{(\log x)^2}\right).
\end{align}
Following Goldfeld, the above estimate can be handled firstly by manipulating the weighted sum $L(x;x^c,x)$. For $33/64<c<1$, from Lemmas \ref{lem1} and \ref{lem2} we clearly have
\begin{align}\label{eq2-2}
L(x;x^c,x)&=L(x;1,x)-L(x;1,x^{1/2})-L(x;x^{1/2},x^c)+O(\log x)\nonumber\\
&=x/2-L(x;x^{1/2},x^c)+O\left(\frac{x\log\log x}{\log x}\right)\nonumber\\
&\le x/2-L(x;x^{1/2},x^{33/64})+O\left(\frac{x\log\log x}{\log x}\right).
\end{align}
We now employ Lemma \ref{lem4} to give a nontrivial lower bound of $L(x;x^{1/2},x^{33/64})$. For integers $1\le j\le \left\lfloor \frac{\log x}{65\log 2}\right\rfloor$, let $Q_j=2^jx^{1/2}$. Then we have $Q_j>x^{1/2}$ and $2Q_j<x^{33/64}$ for sufficiently large $x$. For any $1\le j\le \left\lfloor \frac{\log x}{65\log 2}\right\rfloor$, from Lemma \ref{lem4} (with $A=2$) we have
\begin{align}\label{eq2-3}
\pi(x;m,1)>\frac{K_1(33/64)}{2}\frac{x}{\varphi(m)\log x}
\end{align}
for all integers $m\in(Q_j,2Q_j]$ with at most $O\left(Q_j/(\log x)^2\right)$ exceptions. For $1\le j\le \left\lfloor \frac{\log x}{65\log 2}\right\rfloor$, let $S_j$ be the set of exceptions of $m$ in the interval $(Q_j,2Q_j]$. Thus, in view of equation (\ref{eq2-3}), we have
\begin{align}\label{eq2-4}
L(x;x^{1/2},x^{33/64})&=\sum_{x^{1/2}< m\le x^{33/64}}\Lambda(m)\pi(x;m,1)\nonumber\\
&\ge \sum_{j=1}^{\left\lfloor \frac{\log x}{65\log 2}\right\rfloor}\sum_{\substack{Q_j< m\le 2Q_j\\ m\not\in S_j}}\Lambda(m)\pi(x;m,1)\nonumber\\
&\ge \frac{K_1(33/64)}{2}\frac{x}{\log x}\sum_{j=1}^{\left\lfloor \frac{\log x}{65\log 2}\right\rfloor}\sum_{\substack{Q_j< m\le 2Q_j\\ m\not\in S_j}}\frac{\Lambda(m)}{\varphi(m)}.
\end{align}
For $1\le j\le \left\lfloor \frac{\log x}{65\log 2}\right\rfloor$, we clearly have
$$\sum_{\substack{Q_j< m\le 2Q_j}}\frac{\Lambda(m)}{\varphi(m)}\ge \sum_{\substack{Q_j< p\le 2Q_j}}\frac{\log p}{p}\ge \log Q_j\sum_{\substack{Q_j< p\le 2Q_j}}\frac{1}{p}.$$
Chebyshev's estimate or the prime number theorem gives us
$$\sum_{\substack{Q_j< p\le 2Q_j}}\frac{1}{p}\ge \frac{1}{2Q_j}\sum_{\substack{Q_j< p\le 2Q_j}}1\ge \frac{1}{2Q_j}\frac{Q_j}{2\log Q_j}=\frac{1}{4\log Q_j},$$
from which we deduce that
$$\sum_{\substack{Q_j< m\le 2Q_j}}\frac{\Lambda(m)}{\varphi(m)}\ge \frac{1}{4}.$$
Moreover, we have
$$\sum_{\substack{m\in S_j}}\frac{\Lambda(m)}{\varphi(m)}\ll \log (2Q_j)\sum_{\substack{m\in S_j}}\frac{\log\log m}{m}\ll \log x\log\log x\sum_{\substack{m\in S_j}}\frac{1}{m}$$
since $\varphi(m)\gg \frac{m}{\log\log m}$ (see for example \cite[Theorem 2.9]{MV}), where the implied constants are all absolute. Noticing that $S_j\subset(Q_j,2Q_j]$ and
$|S_j|\ll Q_j/(\log x)^2$ by its definition, we get
$$\sum_{\substack{m\in S_j}}\frac{1}{m}\ll \frac{1}{Q_j}\frac{Q_j}{(\log x)^2}=1/(\log x)^2.$$
It follows that $$\sum_{\substack{m\in S_j}}\frac{\Lambda(m)}{\varphi(m)}\ll \frac{\log\log x}{\log x},$$
where the implied constant is absolute. Thus, we have
\begin{equation}\label{inequality}
\sum_{\substack{Q_j< m\le 2Q_j\\ m\not\in S_j}}\frac{\Lambda(m)}{\varphi(m)}\ge 1/5
\end{equation}
provided that $x$ is sufficiently large.
Combining equation (\ref{eq2-2}), inequalities (\ref{eq2-4}) and (\ref{inequality}) we conclude that there exists some $\delta>0$ such that
\begin{align}\label{eq2-5}
L(x;x^c,x)\le (1/2-\delta)x+O\left(\frac{x\log\log x}{\log x}\right).
\end{align}
The transition from $L(x;x^c,x)$ to $\sum_{x^c\le q\le x}\pi(x;q,1)\log q$ is somewhat standard. It is clear that
\begin{align}\label{eq2-6}
L(x;x^c,x)-\sum_{x^c\le q\le x}\pi(x;q,1)\log q=\sum_{\substack{x^c\le q^k\le x\\k>1}}\pi(x;q^k,1)\log q+O(\log x).
\end{align}
For simplicity, we now require that $3/4<c<1$. Trivial estimates lead to
\begin{align}\label{eq2-7}
\sum_{\substack{x^c\le q^k\le x\\k>1}}\pi(x;q^k,1)\log q&\le\sum_{\substack{x^{3/4}\le q^k\le x\\q\le x^{1/2}\\k>1}}\pi(x;q^k,1)\log q
\ll x^{1/4}\sum_{\substack{x^{3/4}\le q^k\le x\\q\le x^{1/2}\\k>1}}\log q\nonumber\\
&\ll x^{1/4}\sum_{q\le x^{1/2}}\log q\log x\ll x^{3/4}(\log x)^2.
\end{align}
From equations (\ref{eq2-5}), (\ref{eq2-6}) and (\ref{eq2-7}), we obtain
\begin{align}\label{eq2-8}
\sum_{x^c\le q\le x}\pi(x;q,1)\log q\le (1/2-\delta)x+O\left(\frac{x\log\log x}{\log x}\right)
\end{align}
provided that $3/4<c<1$. Hence, equation (\ref{eq2-8}) together with equation (\ref{eq2-1}) yield
$$T_c(x)\le \frac1c(1/2-\delta)x/\log x+O\left(\frac{x\log\log x}{(\log x)^2}\right).$$
Now the theorem follows from taking $c=\max\{1-\delta,4/5\}$.
\end{proof}

\section{Final remarks}
As early as 1980, Pomerance \cite{Po} conjectured that for any $c\in(0,1)$ we have
$$\sum_{\substack{p\le x\\ P^+(p-1)\le x^c}}1\sim \rho(1/c)\pi(x), \quad \text{as~}x\rightarrow\infty.$$
This is surely a rather difficult conjecture which is beyond the power of present mathematics. Granville \cite{Gran} claimed that it follows from the Eillott--Halberstam conjecture. Recently, Wang \cite{Wang} offered an alternative proof of this claim. Therefore, assuming the Eillott--Halberstam conjecture, one can prove $$\lim_{x\rightarrow\infty}T_c(x)/\pi(x)\rightarrow0, \quad \text{as~}c\rightarrow1$$
by Lemma \ref{lem3}. This immediately leads to a negative answer of Conjecture \ref{conje1} for $k=1$ under the Eillott--Halberstam conjecture. So this note can be viewed as an unconditional proof of this matter. The exceptional result for level greater than $1/2$ (see Lemma \ref{lem4}) is crucial for this goal.

\section*{Acknowledgments}
The author would like to thank the anonymous referee for his/her helpful comments which improved the quality of this note greatly. 

The author is supported by National Natural Science Foundation of China under Grant No. 12201544, Natural Science Foundation of Jiangsu Province, China, Grant No. BK20210784, China Postdoctoral Science Foundation, Grant No. 2022M710121, the foundations of the projects "Jiangsu Provincial Double--Innovation Doctor Program'', Grant No. JSSCBS20211023 and "Golden  Phoenix of the Green City--Yang Zhou'' to excellent PhD, Grant No. YZLYJF2020PHD051.

\end{document}